\UseRawInputEncoding
\documentclass[12pt]{article}
\usepackage{amsmath,amsthm,amssymb}
\usepackage{amssymb,latexsym}

\newtheorem{theorem}{Theorem}
\newtheorem{conjecture}{Conjecture}
\newtheorem{lemma}{Lemma}

\begin{document}
\baselineskip=17pt

\title{\bf A quinary diophantine inequality by primes with one of the form $\mathbf{p=x^2+y^2+1}$}

\author{\bf S. I. Dimitrov}

\date{2021}
\maketitle

\begin{abstract}

In this paper we show that, for any fixed $1<c<\frac{5363}{3900}$, every sufficiently large positive number $N$
and a small constant $\varepsilon>0$, the diophantine inequality
\begin{equation*}
|p_1^c+p_2^c+p_3^c+p_4^c+p_5^c-N|<\varepsilon
\end{equation*}
has a solution in prime numbers $p_1,\,p_2,\,p_3,\,p_4,\,p_5$, such that $p_1=x^2 + y^2 +1$.
\quad\\
{\bf  2020 Math.\ Subject Classification}: 11D75  $\cdot$ 11L07 $\cdot$  11L20  $\cdot$  11P32
\end{abstract}

\section{Introduction and statement of the result}
\indent

In 2003 Zhai and Cao \cite{Zhai-Cao1} proved that
for any fixed $1<c<\frac{14142}{8923}$, for every sufficiently large positive number $N$
and $\displaystyle\eta=\frac{1}{\log N}$, the diophantine inequality
\begin{equation}\label{Inequality1}
|p_1^c+p_2^c+p_3^c+p_4^c+p_5^c-N|<\eta
\end{equation}
has a solution in prime numbers $p_1,\,p_2,\,p_3,\,p_4,\,p_5$.

Afterwards the  diophantine inequality \eqref{Inequality1} was improved by  Zhai and Cao \cite{Zhai-Cao2} to
\begin{equation*}
1<c<\frac{81}{40}\,,\quad c\neq2\,, \quad  \eta=\frac{1}{\log N}\,,
\end{equation*}
by Shi and Liu \cite{Shi-Liu} to
\begin{equation*}
1<c<\frac{108}{53}\,,\quad c\neq2\,, \quad  \eta=\frac{1}{\log N}\,,
\end{equation*}
by Baker and Weingartner \cite{Baker-Weingartner} to
\begin{equation*}
2<c<2.041\,,\quad \eta=N^{-\varepsilon}\,,
\end{equation*}
by Li and Cai \cite{Li-Cai} to
\begin{equation*}
2<c<\frac{52}{25}\,,\quad \eta=N^{-\frac{9}{10c}\left(\frac{52}{25}-c\right)}\,,
\end{equation*}
by Zhang and Li \cite{Zhang-Li} to
\begin{equation*}
1<c<\frac{665576}{319965}\,,\quad c\neq2\,,  \quad  \eta=\frac{1}{\log N}\,,
\end{equation*}
by Baker \cite{Baker} to
\begin{equation*}
1<c<\frac{378}{181}\,,\quad c\neq2\,,\quad \eta=N^{-\varepsilon}\,,
\end{equation*}
and this is the best result up to now.

On the other hand in 1960 Linnik \cite{Linnik} showed that there exist infinitely many prime numbers of the form
$p=x^2 + y^2 +1$, where $x$ and $y$ are integers.
More precisely he proved the asymptotic formula
\begin{equation*}
\sum_{p\leq X}r(p-1)=\pi\prod_{p>2}\bigg(1+\frac{\chi_4(p)}{p(p-1)}\bigg)\frac{X}{\log X}+
\mathcal{O}\bigg(\frac{X(\log\log X)^7}{(\log X)^{1+\theta_0}}\bigg)\,,
\end{equation*}
where $r(k)$ is the number of solutions of the
equation $k=x^2 + y^2$ in integers, $\chi_4(k)$ is the non-principal character modulo 4 and
\begin{equation}\label{theta0}
\theta_0=\frac{1}{2}-\frac{1}{4}e\log2=0.0289...
\end{equation}
Motivated by these results we solve \eqref{Inequality1} with prime numbers of a special type.
More precisely we shall prove solvability of \eqref{Inequality1} with Linnik primes.
\begin{theorem}\label{Theorem}
Let $1<c<\frac{5363}{3900}$. For every sufficiently large positive number $N$, the diophantine inequality
\begin{equation*}
|p_1^c+p_2^c+p_3^c+p_4^c+p_5^c-N|<\frac{(\log\log N)^6}{(\log N)^{\theta_0}}
\end{equation*}
has a solution in prime numbers $p_1,\,p_2,\,p_3,\,p_4,\,p_5$, such that $p_1=x^2 + y^2 +1$.
Here $\theta_0$ is defined by \eqref{theta0}.
\end{theorem}

\vspace{1mm}

In addition we have the following task for the future.
\begin{conjecture}  Let $\varepsilon>0$ be a small constant.
There exists $c_0>1$  such that  for any fixed $1<c<c_0$,
and every sufficiently large positive number $N$, the diophantine inequality
\begin{equation*}
|p_1^c+p_2^c+p_3^c+p_4^c+p_5^c-N|<\varepsilon
\end{equation*}
has a solution in prime numbers $p_1,\,p_2,\,p_3,\,p_4,\,p_5$, such that
$p_1=x_1^2 + y_1^2 +1$, $p_2=x_2^2 + y_2^2 +1$, $p_3=x_3^2 + y_3^2 +1$,
$p_4=x_4^2+ y_4^2+1$, $p_5=x_5^2+ y_5^2+1$.
\end{conjecture}

\section{Notations}
\indent

Assume that $N$ is a sufficiently large positive number.
The letter $p$ with or without subscript will always denote prime numbers.
The notation $m\sim M$ means that $m$ runs through the interval $(M/2, M]$.
Moreover $e(t)$=exp($2\pi it$). We denote by  $(m,n)$ the greatest common divisor of $m$ and $n$.
The letter $\eta$ denotes an arbitrary small positive number, not the same in all appearances.
As usual $\varphi (n)$ is Euler's function and $\Lambda(n)$ is von Mangoldt's function.
We shall use the convention that a congruence, $m\equiv n\,\pmod {d}$ will be written as $m\equiv n\,(d)$.
We denote by $r(k)$ the number of solutions of the equation $k=x^2 + y^2$ in integers.
The symbol $\chi_4(k)$ will mean the non-principal character modulo 4.
Throughout this paper unless something else is said, we suppose that $1<c<\frac{5363}{3900}$.

Denote
\begin{align}
\label{X}
&X =\left(\frac{N}{4}\right)^{\frac{1}{c}}\,;\\
\label{D}
&D=\frac{X^{\frac{1}{2}}}{(\log N)^{\frac{6A+34}{3}}}\,,\quad A>3\,;\\
\label{Delta}
&\Delta=X^{\frac{1}{4}-c}\,;\\
\label{varepsilon}
&\varepsilon=\frac{(\log\log X)^6}{(\log X)^{\theta_0}}\,;\\
\label{H}
&H=\frac{\log^2X}{\varepsilon}\,;\\
\label{SldalphaX}
&S_{l,d;J}(t)=\sum\limits_{p\in J\atop{p\equiv l\, (d)}} e(t p^c)\log p\,;\\
\label{SalphaX}
&S(t)=S_{1,1;(X/2,X]}(t)\,;\\
\label{IJalphaX}
&I_J(t)=\int\limits_Je(t y^c)\,dy\,;\\
\label{IalphaX}
&I(t)=I_{(X/2, X]}(t)\,;
\end{align}
\begin{align}
\label{Eytda}
&E(y,t,d,a)=\sum\limits_{\mu y<n\leq y\atop{n\equiv a\, (d)}}\Lambda(n)e(t n^c)
-\frac{1}{\varphi(d)}\int\limits_{\mu y}^{y}e(t x^c)\,dx\,,\\
&\mbox{ where } \quad 0<\mu<1\,.\nonumber
\end{align}

\section{Preliminary lemmas}
\indent

\begin{lemma}\label{Fourier} Let $a, \delta\in \mathbb{R}$ ,
$0 < \delta< a/4$ and $k\in \mathbb{N}$.
There exists a function $\theta(y)$ which is $k$ times continuously differentiable and
such that
\begin{align*}
&\theta(y)=1\quad\quad\quad\mbox{for }\quad\;\; |y|\leq a-\delta\,;\\
&0<\theta(y)<1\quad\,\mbox{for}\quad\quad  a-\delta <|y|< a+\delta\,;\\
&\theta(y)=0\quad\quad\quad\mbox{for}\quad\quad|y|\geq a+\delta\,.
\end{align*}
and its Fourier transform
\begin{equation*}
\Theta(x)=\int\limits_{-\infty}^{\infty}\theta(y)e(-xy)dy
\end{equation*}
satisfies the inequality
\begin{equation*}
|\Theta(x)|\leq\min\Bigg(2a,\frac{1}{\pi|x|},\frac{1}{\pi |x|}
\bigg(\frac{k}{2\pi |x|\delta}\bigg)^k\Bigg)\,.
\end{equation*}
\end{lemma}
\begin{proof}
See (\cite{Shapiro} or \cite{Segal}).
\end{proof}
Throughout this paper we denote by $\theta(y)$ the function from Lemma \ref{Fourier}
with parameters $\displaystyle a = \frac{9\varepsilon}{10}$,
$\displaystyle\delta=\frac{\varepsilon}{10}$, $k=[\log X]$
and by $\Theta(x)$ the Fourier transform of $\theta(y)$.

\smallskip

A key point in the proof of our theorem is the use of the following lemma.
\begin{lemma}\label{Bomb-Vin-Dim} Let $1<c<3$, $c\neq2$, $|t|\leq\Delta$ and $A>0$ be fixed.
Then the inequality
\begin{equation*}
\sum\limits_{d\le \sqrt{X}/(\log N)^{\frac{6A+34}{3}}}\max\limits_{y\le X}
\max\limits_{(a,\, d)=1}\big|E(y,t,d,a)\big|\ll\frac{X}{\log^AX}
\end{equation*}
holds. Here $\Delta$ and $E(y,t,d,a)$ are denoted by \eqref{Delta} and \eqref{Eytda}.
\end{lemma}
\begin{proof}
See (\cite{Dimitrov2}, Lemma 18).
\end{proof}

\begin{lemma}\label{SIasympt} Let $1<c<3$, $c\neq2$ and $|t|\leq\Delta$.
Then for the sum denoted by \eqref{SalphaX} and the integral denoted by \eqref{IalphaX}
the asymptotic formula
\begin{equation*}
S(t)=I(t)+\mathcal{O}\left(\frac{X}{e^{(\log X)^{1/5}}}\right)
\end{equation*}
holds.
\end{lemma}
\begin{proof}
See (\cite{Tolev}, Lemma 14).
\end{proof}

\begin{lemma}\label{intLintI}
For the sum denoted by \eqref{SalphaX} and the integral
denoted by \eqref{IalphaX} we have
\begin{align*}
&\emph{(i)}\quad\quad\quad\;\,
\int\limits_{-\Delta}^\Delta|S(t)|^2\,dt\,\ll X^{2-c}\log^3X\,,
\quad\quad\quad\quad\quad\quad\quad\\
&\emph{(ii)}\quad\quad\quad\int\limits_{-\Delta}^\Delta|I(t)|^2\,dt\ll X^{2-c}\log X\,,\\
\quad\quad\quad\quad\quad\quad\quad
&\emph{(iii)}\quad\quad\;\,
\int\limits_{n}^{n+1}|S(t)|^2\,dt\ll X\log^3X\,.
\quad\quad\quad\quad\quad\quad\quad
\end{align*}
\end{lemma}
\begin{proof}
It follows from the arguments used in (\cite{Tolev}, Lemma 7).
\end{proof}

\begin{lemma}\label{IestTitchmarsh}
Assume that $F(x)$, $G(x)$ are real functions defined in  $[a,b]$,
$|G(x)|\leq H$ for $a\leq x\leq b$ and $G(x)/F'(x)$ is a monotonous function. Set
\begin{equation*}
I=\int\limits_{a}^{b}G(x)e(F(x))dx\,.
\end{equation*}
If $F'(x)\geq h>0$ for all $x\in[a,b]$ or if $F'(x)\leq-h<0$ for all $x\in[a,b]$ then
\begin{equation*}
|I|\ll H/h\,.
\end{equation*}
If $F''(x)\geq h>0$ for all $x\in[a,b]$ or if $F''(x)\leq-h<0$ for all $x\in[a,b]$ then
\begin{equation*}
|I|\ll H/\sqrt h\,.
\end{equation*}
\end{lemma}
\begin{proof} See (\cite{Titchmarsh}, p. 71).
\end{proof}

\begin{lemma}\label{Squareout}
For any complex numbers $a(n)$ we have
\begin{equation*}
\bigg|\sum_{a<n\le b}a(n)\bigg|^2
\leq\bigg(1+\frac{b-a}{Q}\bigg)\sum_{|q|\leq Q}\bigg(1-\frac{|q|}{Q}\bigg)
\sum_{a<n,\, n+q\leq b}a(n+q)\overline{a(n)}\,,
\end{equation*}
where $Q\geq1$.
\end{lemma}
\begin{proof}
See (\cite{Iwaniec-Kowalski}, Lemma 8.17).
\end{proof}

\begin{lemma}\label{Heath-Brown} Let $3 < U < V < Z < X$ and suppose that $Z -\frac{1}{2}\in\mathbb{N}$,
$X\gg Z^2U$, $Z \gg U^2$, $V^3\gg X$.
Assume further that $F(n)$ is a complex valued function such that $|F(n)| \leq 1$.
Then the sum
\begin{equation*}
\sum\limits_{n\sim X}\Lambda(n)F(n)
\end{equation*}
can be decomposed into $O\Big(\log^{10}X\Big)$ sums, each of which is either of Type I
\begin{equation*}
\sum\limits_{m\sim M}a(m)\sum\limits_{l\sim L}F(ml)\,,
\end{equation*}
where
\begin{equation*}
L \gg Z\,, \quad  LM\asymp X\,, \quad |a(m)|\ll m^\eta\,,
\end{equation*}
or of Type II
\begin{equation*}
\sum\limits_{m\sim M}a(m)\sum\limits_{l\sim L}b(l)F(ml)\,,
\end{equation*}
where
\begin{equation*}
U \ll L \ll V\,, \quad  LM\asymp X\,, \quad |a(m)|\ll m^\eta\,,\quad |b(l)|\ll l^\eta\,.
\end{equation*}
\end{lemma}
\begin{proof}
See (\cite{Heath}, Lemma 3).
\end{proof}

\begin{lemma}\label{Exponentpairs}
Let $|f^{(m)}(u)|\asymp YX^{1-m}$  for $1\leq X<u<X_0\leq2X$ and $m\geq1$.\\
Then
\begin{equation*}
\bigg|\sum_{X<n\le X_0}e(f(n))\bigg|
\ll Y^\varkappa X^\lambda +Y^{-1},
\end{equation*}
where $(\varkappa, \lambda)$ is any exponent pair.
\end{lemma}
\begin{proof}
See (\cite{Graham-Kolesnik}, Ch. 3).
\end{proof}

\begin{lemma}\label{ExponentpairofHuxley}
For every $\eta > 0$, the pair $\Big(\frac{32}{205}+\eta, \frac{269}{410}+\eta\Big)$ is an exponent pair.
\end{lemma}
\begin{proof}
See (\cite{Huxley}, Corollary of Theorem 1).
\end{proof}

\begin{lemma}\label{Sargos-Wuest} Let  $\alpha$, $\beta$ be real numbers such that
\begin{equation*}
\alpha\beta(\alpha-1)(\beta-1)(\alpha-2)(\beta-2)\neq0\,.
\end{equation*}
Set
\begin{equation*}
\Sigma_{II}=\sum\limits_{m\sim M}a(m)
\sum\limits_{l\sim L}b(l)e\left(F\frac{m^\alpha l^\beta}{M^\alpha L^\beta}\right)\,,
\end{equation*}
where
\begin{equation*}
F>0\,, \quad M\geq1\,, \quad L\geq1\,, \quad |a(m)|\leq1\,, \quad|b(l)|\leq1\,.
\end{equation*}
Then
\begin{align*}
\Sigma_{II}(FML)^{-\eta}&\ll(F^4M^{31}L^{34})^{\frac{1}{42}}+(F^6M^{53}L^{51})^{\frac{1}{66}}+(F^6M^{46}L^{41})^{\frac{1}{56}}\\
&+(F^2M^{38}L^{29})^{\frac{1}{40}}+(F^3M^{43}L^{32})^{\frac{1}{46}}+(FM^9L^6)^{\frac{1}{10}}\\
&+(F^2M^7L^6)^{\frac{1}{10}}+(FM^6L^6)^{\frac{1}{8}}+M^{\frac{1}{2}}L\\
&+ML^{\frac{1}{2}}+F^{-\frac{1}{2}}ML\,.
\end{align*}
\end{lemma}
\begin{proof}
See (\cite{Sargos-Wu}, Theorem 9).
\end{proof}
The next two lemmas are due to C. Hooley.
\begin{lemma}\label{Hooley1}
For any constant $\omega>0$ we have
\begin{equation*}
\sum\limits_{p\leq X}
\bigg|\sum\limits_{d|p-1\atop{\sqrt{X}(\log X)^{-\omega}<d<\sqrt{X}(\log X)^{\omega}}}
\chi_4(d)\bigg|^2\ll \frac{X(\log\log X)^7}{\log X}\,,
\end{equation*}
where the constant in Vinogradov's symbol depends on $\omega>0$.
\end{lemma}

\begin{lemma}\label{Hooley2} Suppose that $\omega>0$ is a constant
and let $\mathcal{F}_\omega(X)$ be the number of primes $p\leq X$
such that $p-1$ has a divisor in the interval $\big(\sqrt{X}(\log X)^{-\omega}, \sqrt{X}(\log X)^\omega\big)$.
Then
\begin{equation*}
\mathcal{F}_\omega(X)\ll\frac{X(\log\log X)^3}{(\log X)^{1+2\theta_0}}\,,
\end{equation*}
where $\theta_0$ is defined by \eqref{theta0} and the constant in
Vinogradov's symbol depends only on $\omega>0$.
\end{lemma}
The proofs of very similar results are available in (\cite{Hooley}, Ch.5).

\begin{lemma}\label{IIIIest} We have
\begin{equation*}
\int\limits_{-\infty}^{\infty} I^5(t)\Theta(t)e(-Nt)\,dt\gg \varepsilon X^{5-c}\,.
\end{equation*}
\end{lemma}
\begin{proof}
See (\cite{Zhai-Cao1}, Lemma 5).
\end{proof}

\begin{lemma}\label{IntS4}
For the sum denoted by \eqref{SalphaX} we have
\begin{equation*}
\int\limits_{\Delta}^{H}|S(t)|^4|\Theta(t)|\,dt\ll \big(X^{4-c}+X^2\big)X^\eta\,.
\end{equation*}
\end{lemma}

\begin{proof}
See (\cite{Zhai-Cao2}).
\end{proof}

\section{Outline of the proof}
\indent

Consider the sum
\begin{equation}\label{Gamma}
\Gamma(X)=
\sum\limits_{X/2<p_1,p_2,p_3,p_4,p_5\leq X\atop{|p_1^c+p_2^c+p_3^c+p_4^c+p_5^c-N|<\varepsilon}}
r(p_1-1)\prod\limits_{k=1}^{5}\log p_k\,.
\end{equation}
Obviously
\begin{equation}\label{GammaGamma0}
\Gamma(X)\geq\Gamma_0(X)\,,
\end{equation}
where
\begin{equation}\label{Gamma0}
\Gamma_0(X)=\sum\limits_{X/2<p_1,p_2,p_3,p_4,p_5\leq X}r(p_1-1)
\theta\big(p_1^c+p_2^c+p_3^c+p_4^c+p_5^c-N\big)\prod\limits_{k=1}^{5}\log p_k\,.
\end{equation}
From \eqref{Gamma0} and well-known identity $r(n)=4\sum_{d|n}\chi_4(d)$ we obtain
\begin{equation} \label{Gamma0decomp}
\Gamma_0(X)=4\big(\Gamma_1(X)+\Gamma_2(X)+\Gamma_3(X)\big),
\end{equation}
where
\begin{align}
\label{Gamma1}
&\Gamma_1(X)=\sum\limits_{X/2<p_1,p_2,p_3,p_4,p_5\leq X}
\left(\sum\limits_{d|p_1-1\atop{d\leq D}}\chi_4(d)\right)
\theta\big(p_1^c+p_2^c+p_3^c+p_4^c+p_5^c-N\big)\prod\limits_{k=1}^{5}\log p_k\,,\\
\label{Gamma2}
&\Gamma_2(X)=\sum\limits_{X/2<p_1,p_2,p_3,p_4,p_5\leq X}
\left(\sum\limits_{d|p_1-1\atop{D<d<X/D}}\chi_4(d)\right)
\theta\big(p_1^c+p_2^c+p_3^c+p_4^c+p_5^c-N\big)\prod\limits_{k=1}^{5}\log p_k,\\
\label{Gamma3}
&\Gamma_3(X)=\sum\limits_{X/2<p_1,p_2,p_3,p_4,p_5\leq X}
\left(\sum\limits_{d|p_1-1\atop{d\geq X/D}}\chi_4(d)\right)
\theta\big(p_1^c+p_2^c+p_3^c+p_4^c+p_5^c-N\big)\prod\limits_{k=1}^{5}\log p_k.
\end{align}
In order to estimate $\Gamma_1(X)$ and $\Gamma_3(X)$ we need to consider
the sum
\begin{equation} \label{Ild}
I_{l,d;J}(X)=\sum\limits_{X/2<p_2,p_3,p_4,p_5\leq X\atop{p_1\equiv l\,(d)
\atop{p_1\in J}}}\theta\big(p_1^c+p_2^c+p_3^c+p_4^c+p_5^c-N\big)\prod\limits_{k=1}^{5}\log p_k\,,
\end{equation}
where $l$ and $d$ are coprime natural numbers, and $J\subset( X/2,X]$-interval.
If $J=(X/2,X]$ then we write for simplicity $I_{l,d}(X)$.

Using the inverse Fourier transform for the function $\theta(y)$ we deduce
\begin{align*}
I_{l,d;J}(X)&=\sum\limits_{X/2<p_2,p_3,p_4,p_5\leq X\atop{p_1\equiv l\,(d)\atop{p_1\in J}}}
\prod\limits_{k=1}^{5}\log p_k
\int\limits_{-\infty}^{\infty}\Theta(t)e\big((p_1^c+p_2^c+p_3^c+p_4^c+p_5^c-N)t\big)\,dt\\
&=\int\limits_{-\infty}^{\infty}
\Theta(t)S^4(t)S_{l,d;J}(t)e(-Nt)\,dt\,.
\end{align*}
We decompose $I_{l,d;J}(X)$ as follows
\begin{equation}\label{Ilddecomp}
I_{l,d;J}(X)=I^{(1)}_{l,d;J}(X)+I^{(2)}_{l,d;J}(X)+I^{(3)}_{l,d;J}(X)\,,
\end{equation}
where
\begin{align}
\label{Ild1}
&I^{(1)}_{l,d;J}(X)=\int\limits_{-\Delta}^{\Delta}\Theta(t)S^4(t)S_{l,d;J}(t)e(-Nt)\,dt\,,\\
\label{Ild2}
&I^{(2)}_{l,d;J}(X)=\int\limits_{\Delta\leq|t|\leq H}\Theta(t)S^4(t)S_{l,d;J}(t)e(-Nt)\,dt\,,\\
\label{Ild3}
&I^{(3)}_{l,d;J}(X)=\int\limits_{|t|>H}\Theta(t)S^4(t)S_{l,d;J}(t)e(-N t)\,dt\,.
\end{align}
We shall estimate $I^{(1)}_{l,d;J}(X)$, $I^{(3)}_{l,d;J}(X)$,
$\Gamma_3(X),\,\Gamma_2(X)$ and $\Gamma_1(X)$, respectively,
in the sections \ref{SectionIld1}, \ref{SectionIld3},
\ref{SectionGamma3}, \ref{SectionGamma2} and \ref{SectionGamma1}.
In section \ref{Sectionfinal} we shall finalize the proof of Theorem \ref{Theorem}.

\section{Asymptotic formula for $\mathbf{I^{(1)}_{l,d;J}(X)}$}\label{SectionIld1}
\indent

In this section we will derive an asymptotic formula for the integral defined by \eqref{Ild1}.

Denote
\begin{align}
\label{S1}
&S_1=S(t)\,,\\
\label{S2}
&S_2=S_{l,d;J}(t)\,,\\
\label{I1}
&I_1=I(t)\,,\\
\label{I2}
&I_2=\frac{I_J(t)}{\varphi(d)}\,.
\end{align}
We use the identity
\begin{align}\label{Identity}
S^4_1S_2&=I^4_1I_2+(S_2-I_2)I_1^4+S_2(S_1-I_1)I_1^3\nonumber\\
&+S_1S_2(S_1-I_1)I^2_1+S^2_1S_2(S_1-I_1)I_1+S^3_1S_2(S_1-I_1)\,.
\end{align}
We also need the trivial estimations
\begin{equation}\label{trivial estimations}
S_1\ll X \,, \quad S_2\ll\frac{X\log X}{d} \,, \quad  I_1\ll X\,.
\end{equation}
Put
\begin{align}
\label{PhiDeltaJX}
&\Phi_{\Delta, J}(X, d)=\frac{1}{\varphi(d)}
\int\limits_{-\Delta}^{\Delta}\Theta(t)I^4(t)I_J(t)e(-Nt)\,dt\,,\\
\label{PhiJX}
&\Phi_J(X, d)=\frac{1}{\varphi(d)}\int\limits_{-\infty}^{\infty}\Theta(t)I^4(t)I_J(t)e(-Nt)\,dt\,.
\end{align}
Now  \eqref{SalphaX}, \eqref{IalphaX}, \eqref{Ild1}, \eqref{S1} -- \eqref{PhiDeltaJX},
Lemma \ref{Fourier}, Lemma \ref{SIasympt} and Lemma \ref{intLintI} imply
\begin{align}
I^{(1)}_{l,d;J}(X)-\Phi_{\Delta, J}(X, d)
&=\int\limits_{-\Delta}^{\Delta}\Theta(t)\Bigg(S_{l,d;J}(t)-\frac{I_J(t)}{\varphi(d)}\Bigg)
I^4(t)e(-N t)\,dt\nonumber\\
&+\int\limits_{-\Delta}^{\Delta}\Theta(t)S_{l,d;J}(t)
\Big(S(t)-I(t)\Big)I^3(t)e(-Nt)\,dt\nonumber\\
&+\int\limits_{-\Delta}^{\Delta}\Theta(t)S(t)S_{l,d;J}(t)
\Big(S(t)-I(t)\Big)I^2(t)e(-Nt)\,dt\nonumber\\
&+\int\limits_{-\Delta}^{\Delta}\Theta(t)S^2(t)S_{l,d;J}(t)
\Big(S(t)-I(t)\Big)I(t)e(-Nt)\,dt\nonumber\\
&+\int\limits_{-\Delta}^{\Delta}\Theta(t)S^3(t)S_{l,d;J}(t)
\Big(S(t)-I(t)\Big)e(-Nt)\,dt\nonumber\\
&\ll\varepsilon  \Bigg( X^2\max\limits_{|t|\leq\Delta}\bigg|S_{l,d;J}(t)-\frac{I_J(t)}{\varphi(d)}\bigg|
\int\limits_{-\Delta}^{\Delta}|I(t)|^2\,dt\nonumber\\
&+\frac{X^3\log X}{de^{(\log X)^{1/5}}}
\int\limits_{-\Delta}^{\Delta}|I(t)|^2\,dt+\frac{X^3\log X}{de^{(\log X)^{1/5}}}
\int\limits_{-\Delta}^{\Delta}|S(t)|^2\,dt\Bigg)\nonumber
\end{align}
\begin{equation}\label{Ild1-PhiDeltaJX}
\ll\varepsilon\Bigg(X^{4-c}(\log X)\max\limits_{|t|\leq\Delta}
\bigg|S_{l,d;J}(t)-\frac{I_J(t)}{\varphi(d)}\bigg|
+\frac{X^{5-c}}{de^{(\log X)^{1/6}}}\Bigg)\,.
\end{equation}
Using  \eqref{IJalphaX}, \eqref{IalphaX} and Lemma \ref{IestTitchmarsh} we deduce
\begin{equation}\label{IalphaXest}
I_J(t)\ll \min \left(X, \,\frac{X^{1-c}}{|t|}\right)\,,  \quad
I(t)\ll \min \left(X, \,\frac{X^{1-c}}{|t|}\right)\,.
\end{equation}
From \eqref{IJalphaX}, \eqref{IalphaX}, \eqref{PhiDeltaJX}, \eqref{PhiJX}, \eqref{IalphaXest}
and Lemma \ref{Fourier} it follows
\begin{equation*}
\Phi_{\Delta, J}(X, d)-\Phi_J(X, d)\ll \frac{1}{\varphi(d)} \int\limits_{\Delta}^{\infty} |I(t)|^4|I_J(t)| |\Theta(t)|\,dt
\ll \varepsilon\frac{ X^{5-5c}}{\varphi(d)} \int\limits_{\Delta}^{\infty}  \frac{dt}{t^5}
\ll \frac{\varepsilon X^{5-5c}}{\varphi(d)\Delta^4}
\end{equation*}
and therefore
\begin{equation}\label{JXest1}
\Phi_{\Delta, J}(X, d)=\Phi_J(X, d)
+\mathcal{O}\left(\frac{\varepsilon X^{5-5c}}{\varphi(d)\Delta^4}\right)\,.
\end{equation}
Finally  \eqref{Delta}, \eqref{Ild1-PhiDeltaJX}, \eqref{JXest1} and the identity
\begin{equation*}
I^{(1)}_{l,d;J}(X)=I^{(1)}_{l,d;J}(X)-\Phi_{\Delta, J}(X, d)
+\Phi_{\Delta, J}(X, d)-\Phi_J(X, d)+\Phi_J(X, d)
\end{equation*}
yield
\begin{equation}\label{Ild1est}
I^{(1)}_{l,d;J}(X)=\Phi_J(X, d)
+\mathcal{O}\Bigg(\varepsilon X^{4-c}
(\log X)\max\limits_{|t|\leq\Delta}\bigg|S_{l,d;J}(t)-\frac{I_J(t)}{\varphi(d)}\bigg|\Bigg)
+\mathcal{O}\bigg(\frac{\varepsilon X^{5-c}}{de^{(\log X)^{1/6}}}\bigg)\Bigg)\,.
\end{equation}

\section{Upper bound  of $\mathbf{I^{(3)}_{l,d;J}(X)}$}\label{SectionIld3}
\indent

By \eqref{H}, \eqref{SldalphaX}, \eqref{SalphaX},
\eqref{Ild3} and Lemma \ref{Fourier} we find
\begin{equation}\label{Ild3est}
I^{(3)}_{l,d;J}(X)\ll
\frac{X^5\log X}{d}\int\limits_{H}^{\infty}\frac{1}{t}\bigg(\frac{k}{2\pi\delta t}\bigg)^k \,dt
=\frac{X^5\log X}{dk}\bigg(\frac{k}{2\pi\delta H}\bigg)^k
\ll\frac{1}{d}\,.
\end{equation}

\section{Upper bound of $\mathbf{\Gamma_3(X)}$}\label{SectionGamma3}
\indent

Consider the sum  $\Gamma_3(X)$.\\
Since
\begin{equation*}
\sum\limits_{d|p_1-1\atop{d\geq X/D}}\chi_4(d)=\sum\limits_{m|p_1-1\atop{m\leq (p_1-1)D/X}}
\chi_4\bigg(\frac{p_1-1}{m}\bigg)
=\sum\limits_{j=\pm1}\chi_4(j)\sum\limits_{m|p_1-1\atop{m\leq (p_1-1)D/X
\atop{\frac{p_1-1}{m}\equiv j\,(4)}}}1
\end{equation*}
then from \eqref{Gamma3} and \eqref{Ild} we obtain
\begin{equation*}
\Gamma_3(X)=\sum\limits_{m<D\atop{2|m}}\sum\limits_{j=\pm1}\chi_4(j)I_{1+jm,4m;J_m}(X)\,,
\end{equation*}
where $J_m=\big(\max\{1+mX/D,X/2\},X\big]$.
The last formula and \eqref{Ilddecomp} give us
\begin{equation}\label{Gamma3decomp}
\Gamma_3(X)=\Gamma_3^{(1)}(X)+\Gamma_3^{(2)}(X)+\Gamma_3^{(3)}(X)\,,
\end{equation}
where
\begin{equation}\label{Gamma3i}
\Gamma_3^{(i)}(X)=\sum\limits_{m<D\atop{2|m}}\sum\limits_{j=\pm1}\chi_4(j)
I_{1+jm,4m;J_m}^{(i)}(X)\,,\;\; i=1,\,2,\,3.
\end{equation}

\subsection{Estimation of $\mathbf{\Gamma_3^{(1)}(X)}$}
\indent

From \eqref{Ild1est} and \eqref{Gamma3i} we get
\begin{align}\label{Gamma31}
\Gamma_3^{(1)}(X)=\Gamma^*&
+\mathcal{O}\Big(\varepsilon X^{4-c}(\log X)\Sigma_1\Big)
+\mathcal{O}\bigg(\frac{\varepsilon X^{5-c}}{e^{(\log X)^{1/6}}}\Sigma_2\bigg)\,,
\end{align}
where
\begin{align}
\label{Gamma*}
&\Gamma^*=\sum\limits_{m<D\atop{2|m}}
\Phi_J(X, 4m)\sum\limits_{j=\pm1}\chi_4(j)\,,\\
\label{Sigma1}
&\Sigma_1=\sum\limits_{m<D\atop{2|m}}
\max\limits_{|t|\leq\Delta}\bigg|S_{1+jm,4m;J}(t)-\frac{I_J(t)}{\varphi(4m)}\bigg|\,,\\
\label{Sigma2}
&\Sigma_2=\sum\limits_{m<D}\frac{1}{\varphi(4m)}\,.
\end{align}
From the properties of $\chi(k)$ we have that
\begin{equation}\label{Gamma*est}
\Gamma^*=0\,.
\end{equation}
By \eqref{D}, \eqref{SldalphaX},  \eqref{IJalphaX}, \eqref{Sigma1} and Lemma \ref{Bomb-Vin-Dim} we get
\begin{equation}\label{Sigma1est}
\Sigma_1\ll\frac{X}{\log^AX}\,.
\end{equation}
It is well known that
\begin{equation}\label{Sigma2est}
\Sigma_2\ll \log X\,.
\end{equation}
Bearing in mind \eqref{Gamma31}, \eqref{Gamma*est}, \eqref{Sigma1est} and \eqref{Sigma2est} we obtain
\begin{equation}\label{Gamma31est}
\Gamma_3^{(1)}(X)\ll\frac{\varepsilon X^{5-c}}{\log X}\,.
\end{equation}

\subsection{Estimation of $\mathbf{\Gamma_3^{(2)}(X)}$}
\indent

Now we consider $\Gamma_3^{(2)}(X)$. From \eqref{Ild2} and \eqref{Gamma3i} we have
\begin{equation}\label{Gamma32}
\Gamma_3^{(2)}(X)=\int\limits_{\Delta\leq|t|\leq H}\Theta(t)S^4(t)K(t)e(-Nt)\,dt\,,
\end{equation}
where
\begin{equation}\label{Kt}
K(t)=\sum\limits_{m<D\atop{2|m}}\sum\limits_{j=\pm1}\chi_4(j)S_{1+jm,4m;J_m}(t)\,.
\end{equation}

\begin{lemma}\label{IntK2}
For the sum denoted by \eqref{Kt} we have
\begin{equation*}
\int\limits_{\Delta}^{H}|K(t)|^2|\Theta(t)|\,dt\ll X\log^7 X\,.
\end{equation*}
\end{lemma}
\begin{proof}
See (\cite{Dimitrov2}, Lemma 22).
\end{proof}

\begin{lemma}\label{SIest} Assume that
\begin{equation}\label{Conditions1}
\Delta \leq |t| \leq H\,, \quad |a(m)|\ll m^\eta \,,\quad LM\asymp X\,,\quad L\gg X^{\frac{4}{9}} \,.
\end{equation}
Set
\begin{equation}\label{SI}
S_I=\sum\limits_{m\sim M}a(m)\sum\limits_{l\sim L}e(tm^cl^c)\,.
\end{equation}
Then
\begin{equation*}
S_I\ll X^{\frac{1817}{1950}+\eta}\,.
\end{equation*}
\end{lemma}
\begin{proof}
We first consider the case when
\begin{equation}\label{M411}
M\ll X^{\frac{763}{1950}}\,.
\end{equation}
By \eqref{Delta}, \eqref{H}, \eqref{Conditions1}, \eqref{SI}, \eqref{M411} and Lemma \ref{Exponentpairs}
with the exponent pair $\left(\frac{2}{40},\frac{33}{40}\right)$ we obtain
\begin{align}\label{SIest1}
S_I&\ll X^\eta\sum\limits_{m\sim M}\bigg|\sum\limits_{l\sim L}e(tm^cl^c)\bigg|\nonumber\\
&\ll X^\eta\sum\limits_{m\sim M}\Bigg(\big(|t|X^cL^{-1}\big)^{\frac{2}{40}}L^{\frac{33}{40}}+\frac{1}{|t|X^cL^{-1}}\Bigg)\nonumber\\
&\ll X^\eta\Big(H^{\frac{1}{20}}X^{\frac{c}{20}}ML^{\frac{31}{40}}+\Delta^{-1}X^{1-c}\Big)\nonumber\\
&\ll X^\eta\Big(X^{\frac{2c+31}{40}}M^{\frac{9}{40}}+X^{\frac{3}{4}}\Big)\nonumber\\
&\ll X^{\frac{1817}{1950}+\eta}\,.
\end{align}
Next we consider the case when
\begin{equation}\label{M41135}
X^{\frac{763}{1950}}\ll M\ll X^{\frac{5}{9}}\,.
\end{equation}
Using \eqref{SI}, \eqref{M41135} and Lemma \ref{Sargos-Wuest} we deduce
\begin{equation}\label{SIest2}
S_I\ll X^{\frac{1817}{1950}+\eta}\,.
\end{equation}
Bearing in mind \eqref{SIest1} and \eqref{SIest2} we establish the statement in the lemma.
\end{proof}

\begin{lemma}\label{SIIest} Assume that
\begin{equation}\label{Conditions2}
\Delta \leq |t| \leq H\,, \quad |a(m)|\ll m^\eta \,, \quad |b(l)|\ll l^\eta\,,
\quad LM\asymp X\,,\quad X^{\frac{1}{9}} \ll L\ll X^{\frac{1}{3}} \,.
\end{equation}
Set
\begin{equation}\label{SII}
S_{II}=\sum\limits_{m\sim M}a(m)\sum\limits_{l\sim L}b(l)e(tm^cl^c)\,.
\end{equation}
Then
\begin{equation*}
S_{II}\ll X^{\frac{1817}{1950}+\eta}\,.
\end{equation*}
\end{lemma}
\begin{proof}
From \eqref{Conditions2}, \eqref{SII}, Cauchy's inequality and Lemma \ref{Squareout} with $Q=X^{\frac{857}{3900}}$ it follows
\begin{equation}\label{SIIest1}
|S_{II}|^2\ll X^\eta\Bigg(\frac{X^2}{Q}+\frac{X}{Q}\sum\limits_{1\leq q\leq Q}
\sum\limits_{l\sim L}\bigg|\sum\limits_{m\sim M}e\big(f(l, m, q)\big)\bigg|\Bigg)
\end{equation}
where $f(l, m, q)=tm^c\big((l+q)^c-l^c\big)$.
Now \eqref{Delta}, \eqref{H}, \eqref{Conditions2}, \eqref{SIIest1}, Lemma \ref{Exponentpairs} with the exponent pair
$\Big(\frac{32}{205}+\eta, \frac{269}{410}+\eta\Big)$ and Lemma \ref{ExponentpairofHuxley} give us
\begin{align}\label{SIIest2}
S_{II}&\ll X^\eta\Bigg(\frac{X^2}{Q}+\frac{X}{Q}\sum\limits_{1\leq q\leq Q}
\sum\limits_{l\sim L}\bigg(\big(|t|qX^{c-1}\big)^{\frac{32}{205}}M^{\frac{269}{410}}+\frac{1}{|t|qX^{c-1}}\bigg)\Bigg)^{\frac{1}{2}}\nonumber\\
&\ll X^\eta\Bigg(\frac{X^2}{Q}+\frac{X}{Q}
\bigg(H^{\frac{32}{205}}X^{\frac{32(c-1)}{205}}M^{\frac{269}{410}}Q^{\frac{237}{205}}L+\Delta^{-1}X^{1-c}L\log Q\bigg)\Bigg)^{\frac{1}{2}}\nonumber\\
&\ll X^{\frac{1817}{1950}+\eta}\,.
\end{align}

\end{proof}

\begin{lemma}\label{SalphaXest} Let $\Delta \leq |t| \leq H$.
Then for the sum denoted by \eqref{SalphaX} we have
\begin{equation*}
S(t)\ll X^{\frac{1817}{1950}+\eta}\,.
\end{equation*}
\end{lemma}
\begin{proof}
In order to prove the lemma  we will use the formula
\begin{equation}\label{Lambdalog}
S(t)=S^\ast(t)+\mathcal{O}\Big(X^{\frac{1}{2}+\varepsilon}\Big)\,,
\end{equation}
where
\begin{equation*}
S^\ast(t)=\sum\limits_{X/2<n\leq X}\Lambda(n)e(t n^c)\,.
\end{equation*}
Let
\begin{equation*}
U=X^{\frac{1}{9}}\,,\quad V=X^{\frac{1}{3}}\,,\quad Z=\Big[X^{\frac{4}{9}}\Big]+\frac{1}{2}\,.
\end{equation*}
According to Lemma \ref{Heath-Brown}, the sum $S^\ast(t)$
can be decomposed into $O\Big(\log^{10}X\Big)$ sums, each of which is either of Type I
\begin{equation*}
\sum\limits_{m\sim M}a(m)\sum\limits_{l\sim L}e(tm^cl^c)\,,
\end{equation*}
where
\begin{equation*}
L \gg Z\,, \quad  LM\asymp X\,, \quad |a(m)|\ll m^\eta\,,
\end{equation*}
or of Type II
\begin{equation*}
\sum\limits_{m\sim M}a(m)\sum\limits_{l\sim L}b(l)e(tm^cl^c)\,,
\end{equation*}
where
\begin{equation*}
U \ll L \ll V\,, \quad  LM\asymp X\,, \quad |a(m)|\ll m^\eta\,,\quad |b(l)|\ll l^\eta\,.
\end{equation*}
Using Lemma \ref{SIest} and  Lemma \ref{SIIest} we deduce
\begin{equation}\label{Sast}
S^\ast(t)\ll X^{\frac{1817}{1950}+\eta}\,.
\end{equation}
Bearing in mind \eqref{Lambdalog} and  \eqref{Sast} we establish the statement in the lemma.
\end{proof}

\begin{lemma}\label{IntS6}
For the sum denoted by \eqref{SalphaX} we have
\begin{equation*}
\int\limits_{\Delta}^{H}|S(t)|^6|\Theta(t)|\,dt\ll X^{\frac{22469}{3900}-c+\eta}\,.
\end{equation*}
\end{lemma}
\begin{proof}

Denote
\begin{equation}\label{At}
A(t)=\sum\limits_{n\sim X} e(t n^c)
\end{equation}
For any continuous function $\mathfrak{S}(t)$ defined in the interval $[-H, H]$  we have
\begin{align}\label{IntSPsi}
\left|\int\limits_{\Delta\leq|t|\leq H}S(t)\mathfrak{S}(t)\,dt\right|
&=\left|\sum\limits_{p\sim X}(\log p)\int\limits_{\Delta\leq|t|\leq H}e(t p^c)\mathfrak{S}(t)\,dt\right|\nonumber\\
&\leq\sum\limits_{p\sim X}(\log p)\left|\int\limits_{\Delta\leq|t|\leq H}e(t p^c)\mathfrak{S}(t)\,dt\right|\nonumber\\
&\leq(\log X)\sum\limits_{n\sim X}\left|\int\limits_{\Delta\leq|t|\leq H}e(t n^c)\mathfrak{S}(t)\,dt\right|\,.
\end{align}
By \eqref{IntSPsi} and Cauchy's inequality we obtain
\begin{align*}
\left|\int\limits_{\Delta\leq|t|\leq H}S(t)\mathfrak{S}(t)\,dt\right|^2
&\leq X(\log X)^2\sum\limits_{n\sim X}\left|\int\limits_{\Delta\leq|t|\leq H}e(t n^c)\mathfrak{S}(t)\,dt\right|^2\\
\end{align*}
\begin{align}\label{IntSPsi2}
&= X(\log X)^2\int\limits_{\Delta\leq|y|\leq H}\overline{\mathfrak{S}(y)}\,dy
\int\limits_{\Delta\leq|t|\leq H}\mathfrak{S}(t)A(t-y)\,dt\nonumber\\
&\leq X(\log X)^2\int\limits_{\Delta\leq|y|\leq H}|\mathfrak{S}(y)|\,dy
\int\limits_{\Delta\leq|t|\leq H}|\mathfrak{S}(t)||A(t-y)|\,dt\,.
\end{align}
From \eqref{At} and Lemma \ref{Exponentpairs} with the exponent pair
$\left(\frac{1}{2},\frac{1}{2}\right)$ it follows
\begin{equation}\label{Atest}
A(t)\ll\min\bigg(\big(|t|X^{c-1}\big)^{\frac{1}{2}}X^{\frac{1}{2}}+\frac{X^{1-c}}{|t|},\, X\bigg)\,.
\end{equation}
Using  \eqref{H} and \eqref{Atest} we write
\begin{align}\label{IntPsiAt}
&\int\limits_{\Delta\leq|t|\leq H}|\mathfrak{S}(t)||A(t-y)|\,dt\nonumber\\
&\ll\int\limits_{\Delta\leq|t|\leq H\atop{|t-y|\leq X^{-c}}}|\mathfrak{S}(t)||A(t-y)|\,dt
+\int\limits_{\Delta\leq|t|\leq H\atop{X^{-c}<|t-y|\leq2H}}|\mathfrak{S}(t)||A(t-y)|\,dt\nonumber\\
&\ll X\int\limits_{\Delta\leq|t|\leq H\atop{|t-y|\leq X^{-c}}}|\mathfrak{S}(t)|\,dt
+\int\limits_{\Delta\leq|t|\leq H\atop{X^{-c}<|t-y|\leq2H}}|\mathfrak{S}(t)|
\bigg(\big(|t-y|X^{c-1}\big)^{\frac{1}{2}}X^{\frac{1}{2}}+\frac{X^{1-c}}{|t-y|}\bigg)\,dt\nonumber\\
&\ll X\max\limits_{\Delta\leq|t|\leq H}|\mathfrak{S}(t)|\int\limits_{|t-y|\leq X^{-c}}\,dt
+X^{\frac{c}{2}+\eta}\int\limits_{\Delta\leq|t|\leq H}|\mathfrak{S}(t)|\,dt\nonumber\\
&+X^{1-c}\max\limits_{\Delta\leq|t|\leq H}|\mathfrak{S}(t)|\int\limits_{X^{-c}<|t-y|\leq2H}\frac{1}{|t-y|}\,dt\nonumber\\
&\ll X^{1-c}(\log X)\max\limits_{\Delta\leq|t|\leq H}|\mathfrak{S}(t)|
+X^{\frac{c}{2}+\eta}\int\limits_{\Delta\leq|t|\leq H}|\mathfrak{S}(t)|\,dt\,.
\end{align}
Now  \eqref{IntSPsi2} and \eqref{IntPsiAt}  imply
\begin{align}\label{IntSPsi2est}
\left|\int\limits_{\Delta\leq|t|\leq H}S(t)\mathfrak{S}(t)\,dt\right|^2
&\leq X^{2-c+\eta}\max\limits_{\Delta\leq|t|\leq H}|\mathfrak{S}(t)|
\int\limits_{\Delta\leq|t|\leq H}|\mathfrak{S}(t)|\,dt\nonumber\\
&+X^{\frac{c+2}{2}+\eta}\left(\int\limits_{\Delta\leq|t|\leq H}|\mathfrak{S}(t)|\,dt\right)^2\,.
\end{align}
Let's put first
\begin{equation}\label{Psit1}
\mathfrak{S}_1(t)=\overline{S(t)}|S(t)|^3|\Theta(t)|\,.
\end{equation}
Bearing in mind \eqref{varepsilon}, \eqref{IntSPsi2est}, \eqref{Psit1},
Lemma \ref{Fourier}, Lemma \ref{IntS4} and Lemma \ref{SalphaXest} we get
\begin{align}\label{IntS5}
\int\limits_{\Delta\leq|t|\leq H}|S(t)|^5|\Theta(t)|\,dt
&=\int\limits_{\Delta\leq|t|\leq H}S(t)\mathfrak{S}_1(t)\,dt\nonumber\\
&\ll \varepsilon^{\frac{1}{2}}X^{\frac{2792}{975}-\frac{c}{2}+\eta}
\left(\int\limits_{\Delta\leq|t|\leq H}|S(t)|^4|\Theta(t)|\,dt\right)^{\frac{1}{2}}\nonumber\\
&+X^{\frac{c+2}{4}+\eta}\int\limits_{\Delta\leq|t|\leq H}|S(t)|^4|\Theta(t)|\,dt\nonumber\\
&\ll X^{\frac{4742}{975}-c+\eta}+X^{\frac{18-3c}{4}+\eta}\nonumber\\
&\ll X^{\frac{4742}{975}-c+\eta}\,.
\end{align}
Next we put
\begin{equation}\label{Psit2}
\mathfrak{S}_2(t)=\overline{S(t)}|S(t)|^4|\Theta(t)|\,.
\end{equation}
Taking into account \eqref{varepsilon}, \eqref{IntSPsi2est}, \eqref{IntS5}, \eqref{Psit2},
Lemma \ref{Fourier} and Lemma \ref{SalphaXest}  we find
\begin{align*}
\int\limits_{\Delta\leq|t|\leq H}|S(t)|^6|\Theta(t)|\,dt
&=\int\limits_{\Delta\leq|t|\leq H}S(t)\mathfrak{S}_2(t)\,dt\\
&\ll \varepsilon^{\frac{1}{2}}X^{\frac{2597}{780}-\frac{c}{2}+\eta}
\left(\int\limits_{\Delta\leq|t|\leq H}|S(t)|^5|\Theta(t)|\,dt\right)^{\frac{1}{2}}\\
&+X^{\frac{c+2}{4}+\eta}\int\limits_{\Delta\leq|t|\leq H}|S(t)|^5|\Theta(t)|\,dt\\
&\ll X^{\frac{22469}{3900}-c+\eta}+X^{\frac{10459}{1950}-\frac{3c}{4}+\eta}\\
&\ll X^{\frac{22469}{3900}-c+\eta}\,.
\end{align*}
The lemma is proved.
\end{proof}
Finally \eqref{varepsilon}, \eqref{Gamma32}, Cauchy's inequality,
Lemma \ref{IntK2}, Lemma \ref{SalphaXest} and  Lemma \ref{IntS6} yield
\begin{equation}\label{Gamma32est}
\Gamma_3^{(2)}(X)\ll\max\limits_{\Delta\leq t\leq H}|S(t)|
\left(\int\limits_{\Delta}^{H}|S(t)|^6|\Theta(t)|\,dt\right)^{1/2}
\left(\int\limits_{\Delta}^{H}|K(t)|^2|\Theta(t)|\,dt\right)^{1/2}\ll\frac{\varepsilon X^{5-c}}{\log X}\,.
\end{equation}

\subsection{Estimation of $\mathbf{\Gamma_3^{(3)}(X)}$}
\indent

From \eqref{Ild3est} and \eqref{Gamma3i} we have
\begin{equation}\label{Gamma33est}
\Gamma_3^{(3)}(X)\ll\sum\limits_{m<D}\frac{1}{d}\ll \log X\,.
\end{equation}

\subsection{Estimation of $\mathbf{\Gamma_3(X)}$}
\indent

Summarizing \eqref{Gamma3decomp}, \eqref{Gamma31est}, \eqref{Gamma32est} and \eqref{Gamma33est} we get
\begin{equation}\label{Gamm3est}
\Gamma_3(X)\ll\frac{\varepsilon X^{5-c}}{\log X}\,.
\end{equation}

\section{Upper bound of $\mathbf{\Gamma_2(X)}$}\label{SectionGamma2}
\indent

In this section we need the following lemma.
\begin{lemma}\label{Thenumberofsolutions}
Let $1<c<3$, $c\neq2$ and $N_0$ is a sufficiently large positive number.
Then for the number of solutions $B_0(N_0)$ of the  diophantine inequality
\begin{equation}\label{Ternary}
|p_1^c+p_2^c+p_3^c+p_4^c-N_0|<\varepsilon
\end{equation}
in prime numbers $p_1,\,p_2,\,p_3,\,p_4 \in \left((N_0/3)^{\frac{1}{c}}/2\,,\, (N_0/3)^{\frac{1}{c}}\right]$ we have that
\begin{equation*}
B_0(N_0)\ll \frac{\varepsilon N_0^{\frac{4}{c}-1}}{\log^4N_0}\,.
\end{equation*}
\end{lemma}
\begin{proof}
Define
\begin{equation}\label{BX0}
B(X_0)=\sum\limits_{X_0/2<p_1,p_2,p_3,p_4\leq X_0\atop{|p_1^c+p_2^c+p_3^c+p_4^c-N_0|<\varepsilon}}
\log p_1\log p_2\log p_3\log p_4\,,
\end{equation}
where
\begin{equation}\label{X0}
X_0=\left(\frac{N_0}{3}\right)^{\frac{1}{c}}\,.
\end{equation}
We take  the parameters $\displaystyle a_0 = \frac{5\varepsilon}{4}$,
$\displaystyle\delta_0=\frac{\varepsilon}{4}$ and $k_0=[\log X_0]$.
According to Lemma \ref{Fourier} there exists a function $\theta_0(y)$
which is $k_0$ times continuously differentiable and such that
\begin{align*}
&\theta_0(y)=1\quad\quad\quad\mbox{for }\quad\;\; |y|\leq \varepsilon\,;\\
&0<\theta_0(y)<1\quad\,\mbox{for}\quad\quad  \varepsilon <|y|< \frac{3\varepsilon}{2}\,;\\
&\theta_0(y)=0\quad\quad\quad\mbox{for}\quad\quad|y|\geq \frac{3\varepsilon}{2}\,.
\end{align*}
and its Fourier transform
\begin{equation*}
\Theta_0(x)=\int\limits_{-\infty}^{\infty}\theta_0(y)e(-xy)dy
\end{equation*}
satisfies the inequality
\begin{equation}\label{Theta0est}
|\Theta_0(x)|\leq\min\Bigg(\frac{5\varepsilon}{2},\frac{1}{\pi|x|},\frac{1}{\pi |x|}
\bigg(\frac{2k_0}{\pi |x|\varepsilon}\bigg)^{k_0}\Bigg)\,.
\end{equation}
Using  \eqref{BX0}, the definition of $\theta_0(y)$ and  the inverse Fourier
transformation formula we get
\begin{align}\label{BX0est1}
B(X_0)&\leq\sum\limits_{X_0/2<p_1,p_2,p_3,p_4\leq X_0}
\theta_0\big(p_1^c+p_2^c+p_3^c+p_4^c-N_0\big)\log p_1\log p_2\log p_3\log p_4\nonumber\\
&=\int\limits_{-\infty}^{\infty}\Theta_0(t) S_0^4(t) e(-N_0t)\,dt
=B_1(X_0)+B_2(X_0)+B_3(X_0)\,,
\end{align}
where
\begin{align}
\label{S0t}
&S_0(t)=\sum\limits_{X_0/2<p\leq X_0} e(t p^c)\log p\,,\\
\label{Delta0}
&\Delta_0=\frac{(\log X_0)^{A_0}}{X_0^c}\,,\quad A_0>10\,,\\
\label{B1}
&B_1(X_0)=\int\limits_{-\Delta_0}^{\Delta_0}\Theta_0(t)S_0^4(t)e(-N_0t)\,dt\,,\\
\label{B2}
&B_2(X_0)=\int\limits_{\Delta_0\leq|t|\leq H}\Theta_0(t)S_0^4(t)e(-N_0t)\,dt\,,\\
\label{B3}
&B_3(X_0)=\int\limits_{|t|>H}\Theta_0(t)S_0^4(t)e(-N_0 t)\,dt\,.
\end{align}
We begin with the estimation of $B_1(X_0)$. Set
\begin{align}
\label{I0t}
&I_0(t)=\int\limits_{X_0/2}^{X_0}e(t y^c)\,dy\,,\\
\label{PsiDeltaX}
&\Psi_{\Delta_0}(X_0)=\int\limits_{-\Delta_0}^{\Delta_0}\Theta_0(t)I_0^4(t)e(-N_0t)\,dt\,,\\
\label{PsiX}
&\Psi(X_0)=\int\limits_{-\infty}^{\infty}\Theta_0(t)I_0^4(t)e(-N_0t)\,dt\,.
\end{align}
From \eqref{Theta0est}, \eqref{I0t}, \eqref{PsiX} and Lemma \ref{IestTitchmarsh} we obtain
\begin{align}\label{Psiest}
\Psi(X_0)&=\int\limits_{-X_0^{-c}}^{X_0^{-c}}\Theta_0(t)I_0^4(t)e(-N_0t)\,dt
+\int\limits_{|t|>X_0^{-c}}\Theta_0(t)I_0^4(t)e(-N_0t)\,dt\,,\nonumber\\
&\ll\int\limits_{-X_0^{-c}}^{X_0^{-c}}\varepsilon X^4_0\,dt
+\int\limits_{X_0^{-c}}^{\infty}\varepsilon \left(\frac{X_0^{1-c}}{t}\right)^4\,dt\,,\nonumber\\
&\ll\varepsilon X_0^{4-c}\,.
\end{align}
By  \eqref{Theta0est}, \eqref{Delta0}, \eqref{B1}, \eqref{PsiDeltaX},
Lemma \ref{SIasympt} and the trivial estimations
\begin{equation}\label{S0I0trivest}
S_0(t)\ll X_0 \,, \quad I_0(t)\ll X_0
\end{equation}
it follows
\begin{align}\label{B1PsiDelta}
B_1(X_0)-\Psi_{\Delta_0}(X_0)
&\ll\int\limits_{-\Delta_0}^{\Delta_0}|S_0^4(t)-I_0^4(t)||\Theta_0(t)|\,dt\nonumber\\
&\ll\varepsilon\int\limits_{-\Delta_0}^{\Delta_0}
\big|S_0(t)-I_0(t)\big|\Big(|S_0(t)|^3+|I_0(t)|^3\Big)\,dt\nonumber\\
&\ll\varepsilon \frac{X_0}{e^{(\log X_0)^{1/5}}}
\left(\int\limits_{-\Delta_0}^{\Delta_0}|S_0(t)|^3\,dt
+\int\limits_{-\Delta_0}^{\Delta_0}|I_0(t)|^3\,dt\right)\nonumber\\
&\ll\frac{\varepsilon X_0^{4-c}}{e^{(\log X_0)^{1/6}}}\,.
\end{align}
On the other hand \eqref{Theta0est}, \eqref{Delta0}, \eqref{PsiDeltaX},
\eqref{PsiX} and Lemma \ref{IestTitchmarsh} give us
\begin{align}\label{PsiPsiDelta}
|\Psi(X_0)-\Psi_{\Delta_0}(X_0)|&\ll\int\limits_{\Delta_0}^{\infty}|I_0(t)|^4|\Theta_0(t)|\,dt
\ll\frac{\varepsilon}{X_0^{4(c-1)}}\int\limits_{\Delta_0}^{\infty}\frac{dt}{t^4}\nonumber\\
&\ll \frac{\varepsilon}{X_0^{4(c-1)}\Delta^3_0}\ll\frac{\varepsilon X_0^{4-c}}{\log X_0}\,.
\end{align}
From \eqref{Psiest}, \eqref{B1PsiDelta} and \eqref{PsiPsiDelta} and the identity
\begin{equation*}
B_1(X_0)=B_1(X_0)-\Psi_{\Delta_0}(X_0)+\Psi_{\Delta_0}(X_0)-\Psi(X_0)+\Psi(X_0)
\end{equation*}
we deduce
\begin{equation}\label{B1X0est}
B_1(X_0)\ll \varepsilon X_0^{4-c}\,.
\end{equation}
Further we estimate $B_2(X_0)$. Consider the integral
\begin{equation}\label{B2'}
B^\ast_2(X_0)=\int\limits_{\Delta_0}^{H}\Theta_0(t)S_0^4(t)e(-N_0t)\,dt\,.
\end{equation}
Now  \eqref{X0}, \eqref{Theta0est}, \eqref{S0I0trivest}, \eqref{B2'} and partial integration yield
\begin{align}\label{B2'est}
B^\ast_2(X_0)&=-\frac{1}{2\pi i}\int\limits_{\Delta_0}^{H}\frac{\Theta_0(t)S_0^4(t)}{N_0}\,d\,e(-N_0t)\nonumber\\
&=-\frac{\Theta_0(t)S_0^4(t)e(-N_0t)}{2\pi iN_0}\Bigg|_{\Delta_0}^{H}
+\frac{1}{2\pi iN_0}\int\limits_{\Delta_0}^{H}e(-N_0t)\,d\Big(\Theta_0(t)S_0^4(t)\Big)\nonumber\\
&\ll\varepsilon X_0^{4-c}+X_0^{-c}|\Omega|\,,
\end{align}
where
\begin{equation}\label{Omega}
\Omega=\int\limits_{\Delta_0}^{H}e(-N_0t)\,d\Big(\Theta_0(t)S_0^4(t)\Big)\,.
\end{equation}
Next we consider $\Omega$. Set
\begin{equation}\label{Gammat}
\Gamma \, :\, z=f(t)=\Theta_0(t)S_0^4(t)\,,\quad \Delta_0\leq t\leq H\,.
\end{equation}
Now \eqref{Omega} and \eqref{Gammat} imply
\begin{equation}\label{Omegaest1}
\Omega=\int\limits_{\Gamma} e\Big(-N_0f^{-1}(z)\Big)\,dz\,.
\end{equation}
Using \eqref{Theta0est}, \eqref{S0I0trivest}, \eqref{Gammat} and that the integral \eqref{Omegaest1} is independent of path we get
\begin{equation}\label{Omegaest2}
\Omega=\int\limits_{\overline{\Gamma}} e\Big(-N_0f^{-1}(z)\Big)\,dz\ll\int\limits_{\overline{\Gamma}} |dz|
\ll |f(\Delta_0)|+|f(H)| \ll \varepsilon X_0^4\,,
\end{equation}
where $\overline{\Gamma}$ is the line segment connecting the points $f(\Delta_0)$ and $f(H)$.
Bearing in mind \eqref{B2}, \eqref{B2'}, \eqref{B2'est} and \eqref{Omegaest2} we obtain
\begin{equation}\label{B2X0est}
B_2(X_0)\ll \varepsilon X_0^{4-c}\,.
\end{equation}
Finally we estimate $B_3(X_0)$.
Using\eqref{H}, \eqref{Theta0est}, \eqref{S0t}, \eqref{B3}, \eqref{S0I0trivest} we deduce
\begin{equation}\label{B3X0est}
B_3(X_0)\ll X_0^4\int\limits_{H}^{\infty}\frac{1}{t}\bigg(\frac{2k_0}{\pi t\varepsilon}\bigg)^{k_0} \,dt
\ll X_0^4\bigg(\frac{{k_0}}{\pi H\varepsilon}\bigg)^{k_0}\ll1\,.
\end{equation}
Summarizing \eqref{BX0est1}, \eqref{B1X0est}, \eqref{B2X0est} and \eqref{B3X0est} we find
\begin{equation}\label{BX0est}
B(X_0)\ll \varepsilon X_0^{4-c}\,.
\end{equation}
Taking into account \eqref{BX0}, \eqref{X0} and \eqref{BX0est}, for the number of solutions $B_0(N_0)$
of the  diophantine inequality \eqref{Ternary} we establish
\begin{equation*}
B_0(N_0)\ll \frac{\varepsilon N_0^{\frac{4}{c}-1}}{\log^4N_0}\,.
\end{equation*}
The lemma is proved.
\end{proof}

Consider the sum $\Gamma_2(X)$. We denote by $\mathcal{F}(X)$ the set of all primes
$X/2<p\leq X$ such that $p-1$ has a divisor belongs to the interval $(D,X/D)$.
The inequality $xy\leq x^2+y^2$ and  \eqref{Gamma2} give us
\begin{align*}
\Gamma_2(X)^2&\ll(\log X)^{10}\sum\limits_{X/2<p_1,...,p_{10}\leq X
\atop{|p_1^c+p_2^c+p_3^c+p_4^c+p_5^c-N|<\varepsilon
\atop{|p_6^c+p_7^c+p_8^c+p_9^c+p_{10}^c-N|<\varepsilon}}}
\left|\sum\limits_{d|p_1-1\atop{D<d<X/D}}\chi_4(d)\right|
\left|\sum\limits_{t|p_6-1\atop{D<t<X/D}}\chi_4(t)\right|\\
&\ll(\log X)^{10}\sum\limits_{X/2<p_1,...,p_{10}\leq X
\atop{|p_1^c+p_2^c+p_3^c+p_4^c+p_5^c-N|<\varepsilon
\atop{|p_6^c+p_7^c+p_8^c+p_9^c+p_{10}^c-N|<\varepsilon
\atop{p_6\in\mathcal{F}(X)}}}}\left|\sum\limits_{d|p_1-1
\atop{D<d<X/D}}\chi_4(d)\right|^2\,.
\end{align*}
The summands in the last sum for which $p_1=p_6$ can be estimated with
$\mathcal{O}\big(X^{7+\varepsilon}\big)$.\\
Therefore
\begin{equation}\label{Gamma2est1}
\Gamma_2(X)^2\ll(\log X)^{10}\Sigma_0+X^{7+\varepsilon}\,,
\end{equation}
where
\begin{equation}\label{Sigma0}
\Sigma_0=\sum\limits_{X/2<p_1\leq X}\left|\sum\limits_{d|p_1-1
\atop{D<d<X/D}}\chi_4(d)\right|^2\sum\limits_{X/2<p_6\leq X\atop{p_6\in\mathcal{F}(X)
\atop{p_6\neq p_1}}}\sum\limits_{X/2<p_2, p_3, p_4, p_6, p_7, p_8, p_9, p_{10}\leq X
\atop{|p_1^c+p_2^c+p_3^c+p_4^c+p_5^c-N|<\varepsilon
\atop{|p_6^c+p_7^c+p_8^c+p_9^c+p_{10}^c-N|<\varepsilon}}}1\,.
\end{equation}
Now  \eqref{Sigma0} and Lemma \ref{Thenumberofsolutions}  imply
\begin{equation}\label{Sigma0est}
\Sigma_0\ll \frac{X^{8-2c}}{\log^8X}\,\Sigma'_0\,\Sigma''_0\,,
\end{equation}
where
\begin{equation*}
\Sigma'_0=\sum\limits_{X/2<p\leq X}\left|\sum\limits_{d|p-1\atop{D<d<X/D}}\chi_4(d)\right|^2\,,
\quad \Sigma''_0=\sum\limits_{X/2<p\leq X\atop{p\in\mathcal{F}(X)}}1\,.
\end{equation*}
Applying Lemma \ref{Hooley1} we get
\begin{equation}\label{Sigma0'est}
\Sigma'_0\ll\frac{X(\log\log X)^7}{\log X}\,.
\end{equation}
Using Lemma \ref{Hooley2} we obtain
\begin{equation}\label{Sigma0''est}
\Sigma''_0\ll\frac{X(\log\log X)^3}{(\log X)^{1+2\theta_0}}\,,
\end{equation}
where $\theta_0$ is denoted by  \eqref{theta0}.

Finally \eqref{Gamma2est1}, \eqref{Sigma0est},
\eqref{Sigma0'est} and \eqref{Sigma0''est} yield
\begin{equation}\label{Gamma2est2}
\Gamma_2(X)\ll\frac{ X^{5-c}(\log\log X)^5}{(\log X)^{\theta_0}}
=\frac{\varepsilon X^{5-c}}{\log\log X}\,.
\end{equation}

\section{Lower bound  for $\mathbf{\Gamma_1(X)}$}\label{SectionGamma1}
\indent

Consider the sum $\Gamma_1(X)$.
From \eqref{Gamma1}, \eqref{Ild} and \eqref{Ilddecomp} it follows
\begin{equation}\label{Gamma1decomp}
\Gamma_1(X)=\Gamma_1^{(1)}(X)+\Gamma_1^{(2)}(X)+\Gamma_1^{(3)}(X)\,,
\end{equation}
where
\begin{equation}\label{Gamma1i}
\Gamma_1^{(i)}(X)=\sum\limits_{d\leq D}\chi_4(d)I_{1,d}^{(i)}(X)\,,\;\; i=1,\,2,\,3.
\end{equation}

\subsection{Estimation of $\mathbf{\Gamma_1^{(1)}(X)}$}
\indent

First we consider $\Gamma_1^{(1)}(X)$.
Using formula \eqref{Ild1est} for $J=(X/2,X]$, \eqref{Gamma1i}
and treating the reminder term by the same way as for $\Gamma_3^{(1)}(X)$
we find
\begin{equation} \label{Gamma11est1}
\Gamma_1^{(1)}(X)=\Phi(X)\sum\limits_{d\leq D}\frac{\chi_4(d)}{\varphi(d)}
+\mathcal{O}\bigg(\frac{\varepsilon X^{5-c}}{\log X}\bigg)\,,
\end{equation}
where
\begin{equation*}
\Phi(X)=\int\limits_{-\infty}^{\infty}\Theta(t)I^5(t)e(-Nt)\,dt\,.
\end{equation*}
By  Lemma \ref{IIIIest}  we get
\begin{equation}\label{Philowerbound}
\Phi(X)\gg\varepsilon X^{5-c}\,.
\end{equation}
According to \cite{Dimitrov1} we have
\begin{equation}\label{sumchiphi}
\sum\limits_{d\leq D}\frac{\chi_4(d)}{\varphi(d)}=
\frac{\pi}{4}\prod\limits_p \left(1+\frac{\chi_4(p)}{p(p-1)}\right)+\mathcal{O}\Big(X^{-1/20}\Big)\,.
\end{equation}
From \eqref{Gamma11est1} and \eqref{sumchiphi} we obtain
\begin{equation}\label{Gamma11est2}
\Gamma_1^{(1)}(X)=\frac{\pi}{4}\prod\limits_p \left(1+\frac{\chi_4(p)}{p(p-1)}\right) \Phi(X)
+\mathcal{O}\bigg(\frac{\varepsilon X^{5-c}}{\log X}\bigg)+\mathcal{O}\Big(\Phi(X)X^{-1/20}\Big)\,.
\end{equation}
Now \eqref{Philowerbound} and \eqref{Gamma11est2} imply
\begin{equation}\label{Gamma11est3}
\Gamma_1^{(1)}(X)\gg\varepsilon X^{5-c}\,.
\end{equation}

\subsection{Estimation of $\mathbf{\Gamma_1^{(2)}(X)}$}
\indent

Arguing as in the estimation of $\Gamma_3^{(2)}(X)$ we find
\begin{equation} \label{Gamma12est}
\Gamma_1^{(2)}(X)\ll\frac{\varepsilon X^{5-c}}{\log X}\,.
\end{equation}

\subsection{Estimation of $\mathbf{\Gamma_1^{(3)}(X)}$}
\indent

By \eqref{Ild3est} and \eqref{Gamma1i} it follows
\begin{equation}\label{Gamma13est}
\Gamma_1^{(3)}(X)\ll\sum\limits_{m<D}\frac{1}{d}\ll \log X\,.
\end{equation}

\subsection{Estimation of $\mathbf{\Gamma_1(X)}$}
\indent

Summarizing  \eqref{Gamma1decomp}, \eqref{Gamma11est3}, \eqref{Gamma12est} and \eqref{Gamma13est} we obtain
\begin{equation} \label{Gamma1est}
\Gamma_1(X)\gg\varepsilon X^{5-c}\,.
\end{equation}

\section{Proof of the Theorem}\label{Sectionfinal}
\indent

Taking into account \eqref{varepsilon}, \eqref{GammaGamma0}, \eqref{Gamma0decomp},
\eqref{Gamm3est}, \eqref{Gamma2est2} and \eqref{Gamma1est} we deduce
\begin{equation*}
\Gamma(X)\gg\varepsilon X^{5-c}=\frac{X^{5-c}(\log\log X)^6}{(\log X)^{\theta_0}}\,.
\end{equation*}
The last lower bound yields
\begin{equation}\label{Lowerbound}
\Gamma(X) \rightarrow\infty \quad \mbox{ as } \quad X\rightarrow\infty\,.
\end{equation}
Bearing in mind  \eqref{Gamma} and \eqref{Lowerbound} we establish Theorem \ref{Theorem}.

\vskip20pt
\footnotesize
\begin{flushleft}
S. I. Dimitrov\\
Faculty of Applied Mathematics and Informatics\\
Technical University of Sofia \\
8, St.Kliment Ohridski Blvd. \\
1756 Sofia, BULGARIA\\
e-mail: sdimitrov@tu-sofia.bg\\
\end{flushleft}


\begin{thebibliography}{0}

\bibitem{Baker} R. Baker, {\it Some diophantine equations and inequalities with primes},
Funct. Approx. Comment. Math., \textbf{64} (2), (2021), 203 -- 250.

\bibitem{Baker-Weingartner} B. Baker, A. Weingartner, {\it Some applications of the double large sieve},
Monatsh. Math., {\bf170}, (2013), 261 -- 304.

\bibitem{Dimitrov1} S. I. Dimitrov, {\it Diophantine approximation with one prime of the form $p=x^2+y^2+1$},
Lith. Math. J., \textbf{61}, 4, (2021), 445 -- 459.

\bibitem{Dimitrov2}  S. I. Dimitrov, {\it A ternary diophantine inequality by primes with one of the form $p=x^2+y^2+1$},
Ramanujan J., \textbf{59}, 2, (2022), 571 -- 607.

\bibitem{Graham-Kolesnik} S. W. Graham, G. Kolesnik, {\it Van der Corput's Method of Exponential Sums},
Cambridge University Press, New York, (1991).

\bibitem{Heath}  D. R. Heath-Brown, {\it The Piatetski-Shapiro prime number theorem},
J. Number Theory, \textbf{16}, (1983), 242 -- 266.

\bibitem{Hooley} C. Hooley, {\it Applications of sieve methods to the theory of numbers},
Cambridge Univ. Press, (1976).

\bibitem{Huxley} M. N. Huxley, {\it Exponential sums and the Riemann zeta function V},
Proc. Lond. Math. Soc., \textbf{90} (1), (2005), 1 -- 41.

\bibitem{Iwaniec-Kowalski} H. Iwaniec, E. Kowalski, {\it Analytic number theory},
Colloquium Publications, \textbf{53}, Amer. Math. Soc., (2004).

\bibitem{Li-Cai} S. Li, Y. Cai, {\it On a Diophantine inequality involving prime numbers},
Ramanujan J., {\bf 52}, (2020), 163 -- 174.

\bibitem{Linnik} Ju. Linnik, {\it An asymptotic formula in an additive problem of Hardy and Littlewood},
Izv. Akad. Nauk SSSR, Ser.Mat., {\bf24}, (1960), 629 -- 706 (in Russian).

\bibitem{Shapiro} I. Piatetski-Shapiro, {\it On a variant of the Waring-Goldbach problem},
Mat. Sb., {\bf30}, (1952), 105 -- 120, (in Russian).

\bibitem{Sargos-Wu} P. Sargos, J. Wu, {\it Multiple exponential sums with monomials and their applications in number theory},
Acta Math. Hungar., \textbf{87}, (2000), 333 -- 354.

\bibitem{Segal}B. I. Segal, {\it On a theorem analogous to Waring's theorem},
Dokl. Akad. Nauk SSSR (N. S.), {\bf2}, (1933), 47 -- 49, (in Russian).

\bibitem{Shi-Liu} S. Shi, L. Liu, {\it On a Diophantine inequality involving prime powers},
Monatsh. Math., {\bf 169}, (2013), 423 -- 440.

\bibitem{Titchmarsh} E. Titchmarsh, {\it The Theory of the Riemann Zeta-function}
(revised by D. R. Heath-Brown), Clarendon Press, Oxford (1986).

\bibitem{Tolev} D. Tolev, {\it On a diophantine inequality
involving prime numbers}, Acta Arith., {\bf 61}, (1992), 289 -- 306.

\bibitem{Zhai-Cao1} W. Zhai, X. Cao, {\it On a diophantine inequality over primes},
Adv. Math. (China), {\bf 32}(1), (2003), 63 -- 73.

\bibitem{Zhai-Cao2} W. Zhai, X. Cao, {\it On a diophantine inequality over primes (II)},
Monatsh. Math., {\bf 150}, (2007), 173 -- 179.

\bibitem{Zhang-Li} M. Zhang, J. Li, {\it On a Diophantine inequality with five prime variables},
arXiv:1810.09368v1  [math.NT]  22 Oct 2018.

\end{thebibliography}
\end{document}